\documentclass[12pt]{article}
\usepackage[a4paper,margin=27mm,headheight=14pt]{geometry}
\usepackage[T1]{fontenc}
\usepackage{lmodern,microtype,amsmath,amssymb,amsthm,mathtools}
\usepackage{titlesec,fancyhdr}
\usepackage[colorlinks=true,linkcolor=blue,citecolor=blue,urlcolor=blue]{hyperref}
\hypersetup{pdftitle={Uniqueness and stability for inverse potential scattering with multiple wavenumbers},
pdfauthor={Dong Qiu, Xiang Xu, Yue Zhao}}
\titleformat{\section}[runin]{\normalfont\bfseries}{\thesection.}{.55em}{}[.]
\titleformat{\subsection}[runin]{\normalfont\bfseries}{\thesubsection.}{.55em}{}[.]
\titleformat{\subsubsection}[runin]{\normalfont\bfseries}{\thesubsubsection.}{.55em}{}[.]
\titlespacing*{\section}{0pt}{2.5ex plus 1ex minus .2ex}{.75em}
\titlespacing*{\subsection}{0pt}{2ex plus .8ex minus .2ex}{.75em}
\titlespacing*{\subsubsection}{0pt}{1.8ex plus .6ex minus .2ex}{.75em}
\newtheorem{theorem}{Theorem}[section]
\newtheorem{proposition}[theorem]{Proposition}
\newtheorem{lemma}[theorem]{Lemma}

\theoremstyle{definition}
\theoremstyle{remark}
\newcommand{\R}{\mathbb R}

\newcommand{\Sph}{\mathbb S^2}
\newcommand{\Sc}{\mathcal S}
\newcommand{\norm}[1]{\left\|#1\right\|}
\newcommand{\dd}{\,\mathrm d}
\newcommand{\epsN}{\varepsilon_{\mathrm N}}
\newcommand{\epsE}{\varepsilon_{\mathrm E}}
\newcommand{\ip}[1]{\langle#1\rangle}
\newcommand{\ray}[1]{\mathcal R_+ #1}

\numberwithin{equation}{section}
\title{\bfseries\large UNIQUENESS AND STABILITY FOR
INVERSE POTENTIAL SCATTERING
WITH MULTIPLE WAVENUMBERS}
\author{
Dong Qiu\thanks{\raggedright School of Mathematical Sciences, Zhejiang University,
Hangzhou 310058, China.
\mbox{Email: \href{mailto:qiudong@zju.edu.cn}{\texttt{qiudong@zju.edu.cn}}}.}
\and
Xiang Xu\thanks{\raggedright School of Mathematical Sciences, and Center for
Interdisciplinary Applied Mathematics, Zhejiang University,
Hangzhou 310058, China.
\mbox{Email: \href{mailto:xxu@zju.edu.cn}{\texttt{xxu@zju.edu.cn}}}.}
\and
Yue Zhao\thanks{\raggedright School of Mathematics and Statistics, and Key Lab
NAA-MOE, Central China Normal University, Wuhan 430079, China.
\mbox{Email: \href{mailto:zhaoy@ccnu.edu.cn}{\texttt{zhaoy@ccnu.edu.cn}}}.}
}
\date{}
\begin{document}
\maketitle

\begin{abstract}
This paper is concerned with the inverse scattering problem of determining a real-valued potential for the three-dimensional Schr\"odinger equation from multi-wavenumber data.
Under suitable a priori bounds, we establish local increasing stability from near-field Dirichlet data on a fixed sphere for polynomially decaying potentials, and global increasing stability from full far-field amplitudes for exponentially decaying potentials.
The stability estimates consist of a noise term and a remainder term, with the latter decreasing as the upper measured wavenumber increases.
We also obtain global uniqueness from exact near-field data under weaker regularity assumptions.
The analysis combines analytic continuation, weighted resolvent estimates, and half-ray identities whose differences at chord endpoints cancel the unknown exterior contribution.
\end{abstract}



\section{Introduction and main results}\label{sec:intro}
Scattering measurements made at a finite distance are affected by an unknown potential outside the observation region.
For polynomially decaying potentials, we use a half-ray identity to cancel this exterior contribution and establish stability for the potential inside the measurement sphere from a finite band of near-field data.
With exact data, the same half-ray information determines the potential throughout space.
We also establish finite-band far-field stability for exponentially decaying potentials.

We consider the three-dimensional Schr\"odinger equation
\begin{equation}\label{eq:model}
 (-\Delta+V-k^2)u_V=0\quad\text{in }\R^3,\qquad
 u_V(x,k,d)=e^{ikx\cdot d}+u_V^s(x,k,d),
 \quad k>0,\quad d\in\Sph,
\end{equation}
where $V$ is real and $u_V^s$ is the outgoing scattered field.
Let $R>0$ and let $I\subset(0,\infty)$ be a finite wavenumber interval of positive length.
We study the inverse scattering problem of determining $V$ from either the near-field Dirichlet boundary data
\[
 u_V^s(x,k,d),\qquad
 x\in\partial B_R,\quad d\in\Sph,\quad k\in I,
\]
or the far-field pattern
\[
 A_V(k,\theta,d),\qquad
 (\theta,d)\in\Sph\times\Sph,\quad k\in I.
\]
Here $d$ is the incident direction and $\theta$ is the observation direction.
Both measurements are complex-valued.

For the exponentially decaying and Schwartz potentials considered below, the scattered field admits the asymptotic expansion
\[
 u_V^s(x,k,d)
 =\frac{e^{ik|x|}}{|x|}A_V(k,\theta,d)
   +O(|x|^{-2}),
 \qquad \theta=\frac{x}{|x|},
\]
as $|x|\to\infty$, uniformly in $\theta$ for fixed $k$ and $d$.
Thus $A_V$ is the coefficient of the outgoing spherical wave, with the normalization specified in \eqref{eq:ampnormalization}.
The potential need not vanish outside $B_R$, and its exterior tail is unknown.
For the stability estimates, we take $I=[K_0,K]$, where $K_0$ is above a threshold determined by the potential class.
The outgoing condition is specified in Section~\ref{sec:weighted}.

Throughout, we use the notation $\ip{x}=(1+|x|^2)^{\frac{1}{2}}$, $B_R=\{x\in\R^3:|x|<R\}$, and $q=V_1-V_2$.
Our Fourier transform convention is $\widehat f(\xi)=\int_{\R^3}e^{-ix\cdot\xi}f(x)\dd x$.
We use $H^s(B_R)$ as a restriction space with its quotient norm.
All constants in the stability estimates are uniform over the stated potential class and independent of the upper wavenumber $K$.

\subsection{Near-field stability and global uniqueness}
For an integer $j\ge0$, define
\begin{equation}\label{nf:reg:norminteger}
 \norm V_{C^j_m}
 =\sum_{|\gamma|\le j}\sup_{x\in\R^3}
       \ip{x}^{m}|\partial^\gamma V(x)|.
\end{equation}
For $0<\beta\le1$, set
\begin{equation}\label{nf:reg:normholder}
 \norm V_{C^{2,\beta}_m}
 =\norm V_{C^2_m}
  +\sum_{|\gamma|=2}\sup_{\substack{x\in\R^3\\0<|h|\le1}}
    \ip{x}^{m}
    \frac{|\partial^\gamma V(x+h)-\partial^\gamma V(x)|}{|h|^\beta}.
\end{equation}
The derivatives are classical.
For the measured band, put
\begin{equation}\label{nf:eq:epsilon}
 \epsN=\sup_{K_0\le k\le K}\sup_{d\in\Sph}
 k\norm{u_{V_1}^s(\cdot,k,d)-u_{V_2}^s(\cdot,k,d)}_{L^2(\partial B_R)}.
\end{equation}
The factor $k$ matches the high-frequency trace size and the bounded analytic field used for continuation.
If $\eta_{\rm N}$ denotes the unweighted uniform trace discrepancy on the same band, then $K_0\eta_{\rm N}\le\epsN\le K\eta_{\rm N}$.

\begin{theorem}\label{intro:thm:nearfield-unique}
Fix $m>\frac{5}{2}$ and $R,Q>0$.
There is $K_0\ge2$, depending only on $m,R,Q$, such that the following holds.
Let $V_1,V_2$ be real with $\norm{V_j}_{C^1_m}\le Q$.
If their scattered Dirichlet traces on $\partial B_R$ agree for every $d\in\Sph$ and every $k$ in an interval of positive length contained in $(K_0,\infty)$, then $V_1=V_2$ throughout $\R^3$.
\end{theorem}

\begin{theorem}\label{intro:thm:nearfield}
Fix $m>\frac{5}{2}$ and $R,Q>0$, and choose $K_0$ as in Theorem~\ref{intro:thm:nearfield-unique}.
Fix $0<\beta\le1$ and $s,Q_s>0$.
There are $C>0$, $M\ge1$, and $\varepsilon_*>0$, depending only on $m,R,Q,\beta,s,Q_s$, such that for real potentials satisfying
\[
 \norm{V_j}_{C^{2,\beta}_m}\le Q\quad(j=1,2),\qquad
 \norm{q|_{B_R}}_{H^s(B_R)}\le Q_s,
\]
and for $K\ge K_0+1$, $0<\epsN<\varepsilon_*$, one has
\begin{equation}\label{intro:eq:nearfield}
 \norm q_{L^2(B_R)}^2
 \le C\left\{K^{\frac{2}{1+2s}}\epsN^2+
   \left(\frac{\log(e+KL_{\mathrm N})}{KL_{\mathrm N}}\right)^{%
              \frac{4s}{1+2s}}\right\},
 \qquad L_{\mathrm N}=\log\frac{M}{\epsN}.
\end{equation}
\end{theorem}

The quantitative estimate controls $q$ inside the measurement sphere.
The global exact-data conclusion requires only the weighted $C^1$ bound and has no Sobolev prior on the difference.
The two theorems use the same modulated field.
Its half-ray limit supplies a chord identity for local stability and a Riesz potential identity for global uniqueness.

\subsection{Far-field stability for exponential tails}
For $\mu,Q>0$, let
\[
 \mathcal E_{\mu,Q}
 =\{V\in L^\infty(\R^3):V\text{ is real},\
       |V(x)|\le Qe^{-\mu\ip{x}}\text{ a.e.}\}.
\]
The amplitude is normalized by
\begin{equation}\label{eq:ampnormalization}
 A_V(k,\theta,d)=-\frac{1}{4\pi}\int_{\R^3}
        e^{-ik\theta\cdot x}V(x)u_V(x,k,d)\dd x.
\end{equation}
It is continuous for the exponential class and for Schwartz potentials.
For the quantitative result, define
\begin{equation}\label{ff:eq:epsilon}
 \epsE=\sup_{\substack{k\in[K_0,K]\\\theta,d\in\Sph}}
       |A_{V_1}(k,\theta,d)-A_{V_2}(k,\theta,d)|.
\end{equation}

\begin{theorem}
\label{ff:thm:exponential}
Fix $\mu,s,Q,Q_s>0$.
Choose a fixed $K_0\ge2$ above a threshold depending only on these parameters.
There are $C>0$, $M_{\rm e}\ge1$, and $\varepsilon_{\rm e,*}>0$ such that, for $V_1,V_2\in\mathcal E_{\mu,Q}$ with $\norm q_{H^s(\R^3)}\le Q_s$, $K\ge K_0+1$, and $0<\epsE<\varepsilon_{\rm e,*}$,
\begin{equation}\label{ff:eq:exp-main}
 \norm q_{L^2(\R^3)}^2
 \le C\left\{K^{\frac{6}{3+2s}}\epsE^2
       +\left(K+\log L_{\rm e}\right)^{-\frac{4s}{3+2s}}\right\},
 \qquad L_{\rm e}=\log\frac{M_{\rm e}}{\epsE}.
\end{equation}
The constants and error threshold are independent of $K$ and of the potential pair.
Zero discrepancy implies $V_1=V_2$ almost everywhere in $\R^3$.
\end{theorem}

This estimate is global and imposes no spatial derivative bound on the individual potentials.
Exponential decay gives a common analytic strip for their amplitudes.
Propagation of the measured error through this strip yields the double logarithm in \eqref{ff:eq:exp-main}.

At fixed $K$, both stability bounds tend to zero with their data errors.
The stability estimates consist of a noise term and a remainder term, with the latter decreasing as the upper measured wavenumber increases.

\subsection{Previous work}
Multi-frequency inverse scattering and increasing stability are reviewed by Bao, Li, Lin, and Triki \cite{BaoLiLinTriki2015}.
Related potential estimates include fixed-energy scattering stability \cite{Stefanov1990}, instability on an energy interval \cite{Isaev2013}, and increasing stability for the Schr\"odinger boundary map \cite{Isakov2011,IsakovNagayasuUhlmannWang2014}.
Increasing stability and high-wavenumber reconstruction for the linearized inverse Schr\"odinger potential problem were studied
in \cite{IsakovLuXu2020}.

Our work builds on the work of  Bao, Zhai, and Zhao \cite{BZZ}.
Their compactly supported potential analysis combines resolvent bounds, quantitative continuation in the wavenumber, and high-frequency scattering remainders for near-field and far-field measurements.
The strip-continuation lemma used in \cite[Lemma~2.3]{BZZ} comes from Li, Zhai, and Zhao \cite[Lemma~A.1]{LiZhaiZhao2020}; the underlying two-constants argument is classical harmonic measure.
We extend this quantitative strategy to independently unknown noncompact tails.
In \cite[Theorem~2.2]{BZZ}, the compactly supported potentials have individual $H^s(\R^3)$ and $L^\infty(\R^3)$ bounds.
The near-field discrepancy combines the wavenumber-weighted $L^2$ Dirichlet trace with the Neumann trace obtained through the free exterior Dirichlet-to-Neumann map; the measured input is the Dirichlet trace.
Here the near-field estimate uses only an $L^2$ Dirichlet discrepancy, with its frequency normalization in \eqref{nf:eq:epsilon}, and polynomially weighted $C^{2,\beta}$ bounds on the individual potentials.
Our far-field estimate retains bounded individual potentials but replaces compact support by exponential decay and uses a whole-space Sobolev bound on the difference.

The unknown exterior tail creates a specific difficulty in extending \cite{BZZ}.
Outside the measurement sphere, the scattered field satisfies
\[
 (-\Delta-k^2)u_V^s=-V\bigl(e^{ikx\cdot d}+u_V^s\bigr).
\]
The right-hand side need not vanish and contains the unknown exterior potential.
Consequently, the free exterior Dirichlet-to-Neumann map cannot recover the physical Neumann trace from the measured Dirichlet trace.
Truncating the potential outside the sphere would also change the measured field.
Moreover, without compact support, continuation to complex wavenumbers requires estimates that control the tails.

Noncompact tails arise in atomic interactions and screened potentials, and enter scattering calculations in cold-atom vacuum metrology \cite{GribakinFlambaum1993,EdwardsEtAl2017,BarkerEtAl2022}.

For near-field data, we continue a modulated Dirichlet field in the wavenumber and recover its half-ray limit.
Subtracting the half-rays at the two endpoints of a chord cancels the exterior contribution.
The resulting negative Sobolev estimate, combined with a local positive-order prior, gives stability inside the sphere.
For exact data, the sum of the opposite half-rays gives a Riesz potential identity.
The classical antilocality theorem \cite[Theorem~1.2]{GhoshSaloUhlmann2020} then gives uniqueness throughout space.
Thus the same near-field information supplies both interior stability and global exact-data uniqueness, although the exterior tails of the two potentials may differ.
For far-field data, exponential weights provide a common analytic strip for the individual amplitudes and yield whole-space finite-band stability.
These two estimates are the main quantitative contributions.

Qualitative fixed-energy uniqueness for exponentially decreasing potentials was established by Novikov \cite{Novikov1994}, by Eskin and Ralston for electric and magnetic potentials \cite{EskinRalston1995}, and by Uhlmann and Vasy through a density argument \cite{UhlmannVasyExp}.
Our exponential far-field theorem contributes a quantitative finite-band estimate with explicit error and wavenumber dependence.
The question whether finite-band amplitudes determine arbitrary real Schwartz potentials, stated in \cite[Question~1.4.7]{FeldmanSaloUhlmann2025}, remains outside the present results.

Section~\ref{sec:weighted} establishes the complex-frequency field estimates.
Section~\ref{sec:nearfield} proves near-field stability and global uniqueness.
Section~\ref{ff:sec:exponential} treats far-field stability.
Section~\ref{sec:questions} discusses conditional far-field uniqueness and the remaining questions.
Appendices contain the weighted kernel and difference-quotient proofs and the half-ray profiles used in the near-field argument.

\section{Complex-frequency estimates}\label{sec:weighted}
We construct the analytic field used for the near-field argument under the weighted $C^1_m$ bound.
Write $|D|=(-\Delta)^{\frac{1}{2}}$.
The outgoing free resolvent has kernel
\begin{equation}\label{eq:free-kernel}
 R_0(z;x,y)=\frac{e^{iz|x-y|}}{4\pi|x-y|}.
\end{equation}
The scattered field is selected by the outgoing limiting-absorption condition
\begin{equation}\label{eq:outgoing-condition}
 \begin{aligned}
 u_V^s&=-R_0(k+i0)(Vu_V),\\
 R_0(k+i0)&=\lim_{\eta\downarrow0}
       \bigl(-\Delta-(k+i\eta)^2\bigr)^{-1}.
 \end{aligned}
\end{equation}
The resolvent limit is taken in polynomially weighted $L^2$ spaces.
We use this formulation of the outgoing radiation condition for the polynomially decaying potentials.
For exponentially decaying and Schwartz potentials, it agrees with the classical Sommerfeld radiation condition
\begin{equation}\label{eq:sommerfeld}
 \lim_{r\to\infty}r\bigl(\partial_r u_V^s(r\hat x,k,d)
                    -ik u_V^s(r\hat x,k,d)\bigr)=0,
 \qquad \hat x\in\Sph,
\end{equation}
uniformly in $\hat x$, where $r=|x|$ and $\partial_r$ is the radial derivative.
For real $k>0$, we write $R_0(k)=R_0(k+i0)$.
The physical scattered field satisfies
\begin{equation}\label{eq:ls}
 u_V^s=-R_0(k)\bigl(V(e^{ikx\cdot d}+u_V^s)\bigr).
\end{equation}

\subsection{Polynomial weights and modulated fields}
Fix $m>\frac{5}{2}$ and choose $p=p(m)$ once so that
\[
 1<p<\min\left\{\frac{3}{2},\frac{m}{2},m-\frac{3}{2}\right\},\qquad w(x)=\ip{x}^{-p}.
\]
Then $w\in L^3(\R^3)$, and
\begin{equation}\label{nf:weight:translation}
 C_w^{-1}w(x)\le w(x+h)\le C_w w(x),\qquad |h|\le1.
\end{equation}
The three-dimensional strong Huygens principle gives
\[
 \mathbf1_{B_{\frac{t}{4}}}\cos(t|D|)\mathbf1_{B_{\frac{t}{4}}}=0.
\]
Splitting each weight into interior and exterior parts therefore yields
\begin{equation}\label{nf:weight:wave}
 \norm{w\cos(t|D|)w}_{2\to2}\le2\ip{\frac{t}{4}}^{-p},\qquad
 J_w:=\int_0^\infty\norm{w\cos(t|D|)w}_{2\to2}\dd t<\infty.
\end{equation}

Let $V$ be real with $\norm V_{C^1_m}\le Q$, and put $v=\frac{V}{w^2}$, $f_V=\frac{V}{w}$.
The choices $2p<m$ and $m-p>\frac{3}{2}$ imply
\begin{equation}\label{nf:eq:weightedclass}
 \norm v_\infty+\norm{f_V}_2
 +\sum_{j=1}^3\left(
   \norm{\frac{\partial_jV}{w^2}}_\infty+\norm{\frac{\partial_jV}{w}}_2\right)
 \le C_m Q.
\end{equation}
In particular, $V\in L^{\frac{6}{5}}\cap L^2\cap L^\infty$.

For $d\in\Sph$, define the nonnegative delay and weighted kernel
\[
 \tau_d(x,y)=|x-y|+(y-x)\cdot d\ge0,\qquad
 \mathcal B_d(z;x,y)
   =w(x)\frac{e^{iz\tau_d(x,y)}}{4\pi|x-y|}w(y).
\]
The sign of this delay controls the kernel throughout the upper half-plane, even though complex plane-wave multiplication itself need not be bounded on polynomially weighted spaces.
\begin{lemma}\label{nf:weight:kernel}
The operators $\mathcal B_d(z)$ are bounded and holomorphic in Hilbert--Schmidt norm for $\operatorname{Im}z>0$ and are continuous in that norm to the real axis.
Uniformly in $d$ and $\operatorname{Im}z\ge0$,
\begin{equation}\label{nf:weight:highenergy}
 \norm{\mathcal B_d(z)}_{L^2\to L^2}
 \le\frac{J_w}{|z|},\qquad z\ne0.
\end{equation}
\end{lemma}

The proof in Appendix~\ref{app:weighted} first obtains the real-axis bound from the wave formula.
A regularized maximum principle then gives the same decay for the modulated kernel throughout the upper half-plane.

Choose $\kappa>1$ sufficiently large in terms of $m,Q$ and write $\mathcal D_\kappa=\{z:\operatorname{Im}z>0,\ |z|>\kappa\}$.
The choice is made once to cover translations of $V$ by vectors of length at most one.
Lemma~\ref{nf:weight:kernel} makes $I+\mathcal B_d(z)v$ invertible by a uniform Neumann series on the closure of this domain.
Define
\begin{equation}\label{nf:weight:field}
 g_{V,d}(z)=-(I+\mathcal B_d(z)v)^{-1}\mathcal B_d(z)f_V,
 \qquad a_{V,d}=w^{-1}g_{V,d}.
\end{equation}
Then $\norm{g_{V,d}(z)}_2\le C|z|^{-1}$.
The function $a_{V,d}$ satisfies the modulated integral equation with kernel $\frac{e^{iz\tau_d(x,y)}}{4\pi|x-y|}$.

For real $k>\kappa$, undoing the unitary modulation gives the outgoing limiting-absorption solution and identifies
\begin{equation}\label{nf:weight:physical}
 a_{V,d}(k,x)=e^{-ikx\cdot d}u_V^s(x,k,d).
\end{equation}
The weighted incident and scattered sources are, respectively,
\[
 f_Ve^{ikx\cdot d}\in L^2,\qquad
 e^{ikx\cdot d}vg_{V,d}(k)\in L^2.
\]
The boundary value of the free resolvent selects the outgoing solution, and the Neumann inverse gives uniqueness in this weighted limiting-absorption class.
For nonreal $z$, \eqref{nf:weight:field} defines the analytic continuation directly.

\begin{lemma}\label{nf:weight:trace}
Assume $\norm V_{C^1_m}\le Q$.
For every bounded smooth set $U\subset\R^3$,
\begin{equation}\label{nf:weight:Hone}
 \norm{a_{V,d}(z)}_{H^1(U)}\le C_U|z|^{-1},
 \qquad z\in\overline{\mathcal D_\kappa}.
\end{equation}
The field is holomorphic into $H^1(U)$ on $\mathcal D_\kappa$.
It is continuous to the real boundary in $H^t(U)$ for every $0<t<1$, and the trace on $\partial B_R$ is continuous there in $L^2(\partial B_R)$ whenever $U$ contains that sphere.
If $b_{V,d}=-w^{-1}\mathcal B_d(z)f_V$ is the modulated first Born term, then
\begin{equation}\label{nf:weight:nonlinear}
 \norm{a_{V,d}(z)-b_{V,d}(z)}_{H^1(U)}
 \le C_U|z|^{-2},\qquad z\in\overline{\mathcal D_\kappa}.
\end{equation}
All constants are uniform in the stated class and domain.
\end{lemma}

The difference-quotient proof is included in Appendix~\ref{app:weighted}.

\section{Near-field stability and global uniqueness}\label{sec:nearfield}
We first prove Theorem~\ref{intro:thm:nearfield-unique} under the weighted $C^1_m$ assumption and then prove Theorem~\ref{intro:thm:nearfield}.
All fields use the weight chosen in Section~\ref{sec:weighted}.
Choose $K_0>2\kappa$, depending only on $m,R,Q$, for both theorems.
The $C^1_m$ alternative in Proposition~\ref{nf:prop:ray} is used for uniqueness; the quantitative estimates in Section~\ref{nf:sec:stability-proof} use the stronger assumptions of Theorem~\ref{intro:thm:nearfield}.

\subsection{The half-ray limit}
For bounded $x$ and $d\in\Sph$, define
\[
 \ray V(x,d)=\int_0^\infty V(x-td)\dd t,
 \qquad
 b_{V,d}(z,x)=-\int_{\R^3}
     \frac{e^{iz\tau_d(x,y)}}{4\pi|x-y|}V(y)\dd y.
\]
Polynomial decay makes the half-rays absolutely convergent and continuous in $(x,d)$, uniformly for bounded $x$.
The Born integral is absolutely convergent for $\operatorname{Im}z\ge0$: boundedness handles its local singularity and $V\in L^{\frac{6}{5}}$ handles the exterior Newton kernel.

\begin{proposition}\label{nf:prop:ray}
If $\norm V_{C^{2,\beta}_m}\le Q$, $0<\beta\le1$, then
\begin{equation}\label{nf:eq:ray}
 \sup_{d\in\Sph}
 \norm{2iz\,a_{V,d}(z,\cdot)-\ray V(\cdot,d)}_{L^2(\partial B_R)}
 \le C|z|^{-1},\qquad z\in\overline{\mathcal D_\kappa}.
\end{equation}
If only $\norm V_{C^1_m}\le Q$, then
\begin{equation}\label{nf:eq:imaginary-ray}
 \sup_{d\in\Sph}
 \norm{-2A\,a_{V,d}(iA,\cdot)-\ray V(\cdot,d)}_{L^2(\partial B_R)}
 \le CA^{-\frac{1}{2}},\qquad A\ge\kappa.
\end{equation}
\end{proposition}
\begin{proof}
Lemma~\ref{nf:reg:proposition} in Appendix~\ref{app:rays} gives the corresponding estimates for $b_{V,d}$.
Lemma~\ref{nf:weight:trace} and the trace theorem give $\norm{a_{V,d}-b_{V,d}}_{L^2(\partial B_R)}\le C|z|^{-2}$.
Multiplying by $2iz$ adds $O(|z|^{-1})$, which is bounded by $CA^{-\frac{1}{2}}$ on the imaginary axis.
Since $2i(iA)=-2A$, this proves both formulas.
\end{proof}
\subsection{Chord cancellation and the Riesz potential identity}\label{nf:sec:chords}

Set $q=V_1-V_2$ and $f=\mathbf1_{B_R}q$.
For $d\in\Sph$ and $p\in d^\perp$, define the full X-ray transform by
\[
 Xf(p,d)=\int_\R f(p+td)\dd t.
\]
Since $f$ vanishes outside $B_R$, this transform is zero for $|p|\ge R$.
For $|p|<R$, put
\[
 \ell(p)=\sqrt{R^2-|p|^2},\qquad x_\pm=p\pm\ell(p)d.
\]
Then the exact identity
\begin{equation}\label{nf:eq:chord}
 \ray q(x_+,d)-\ray q(x_-,d)
 =\int_{-\ell(p)}^{\ell(p)}q(p+td)\dd t=Xf(p,d)
\end{equation}
cancels the external half-ray.
It does not assume that $V_1$ and $V_2$ agree outside the ball.

The Fourier-slice identity in dimension three gives
\begin{equation}\label{nf:eq:Xidentity}
 \int_{\Sph}\int_{d^\perp}|Xf|^2\dd p\dd d
 =c_X\int_{\R^3}\frac{|\widehat f(\xi)|^2}{|\xi|}\dd\xi,
 \qquad c_X>0.
\end{equation}
It follows from Plancherel on $d^\perp$ and $\int_{\Sph}\delta(\xi\cdot d)\dd d=\frac{2\pi}{|\xi|}$.
Approximation extends it from smooth compactly supported functions to the present bounded $f$ supported in $B_R$.

If $\ray q(x,d)=0$ for every $x\in\partial B_R$ and $d\in\Sph$, the sum of the opposite half-rays gives
\[
 \int_\R q(x+td)\dd t=\ray q(x,d)+\ray q(x,-d)=0
\]
on every line meeting the sphere.
Every line through $x\in B_R$ meets that sphere, so integration over directions yields
\begin{equation}\label{nf:eq:riesz-zero}
 0=\int_{\Sph}\int_\R q(x+td)\dd t\dd d
   =2\int_{\R^3}\frac{q(y)}{|x-y|^2}\dd y,
 \qquad x\in B_R.
\end{equation}
The $C^1_m$ bound gives $q\in L^{\frac{6}{5}}\cap L^2\cap L^\infty$.
The convolution is absolutely convergent: boundedness controls its local singularity, while the exterior kernel belongs to $L^6$ and pairs with $q\in L^{\frac{6}{5}}$.
The same estimate justifies Fubini in \eqref{nf:eq:riesz-zero}.

\begin{equation}\label{nf:eq:P}
 P_d(z)=z\bigl(a_{V_1,d}(z)-a_{V_2,d}(z)\bigr)\big|_{\partial B_R}.
\end{equation}
Lemma~\ref{nf:weight:trace} makes these functions bounded and holomorphic with values in $L^2(\partial B_R)$, and continuous on the real boundary.
One constant $M\ge1$ bounds their norms on $\mathcal D_\kappa$.
\subsection{Exact-data uniqueness throughout space}\label{nf:sec:unique}
\begin{proof}[Proof of Theorem~\ref{intro:thm:nearfield-unique}]
Assume $\norm{V_j}_{C^1_m}\le Q$ and equality of the near-field data on an interval of positive length in $(K_0,\infty)$.
For fixed $d$ and any continuous complex-linear functional $\ell$ on $L^2(\partial B_R)$, the scalar function $\ell(P_d(z))$ is holomorphic on $\mathcal D_\kappa$, continuous to the measured interval, and zero there.
Choose an interior point of that interval and a small disk centered there whose closure avoids $|z|\le\kappa$ and whose real diameter lies in the interval.
Local Schwarz reflection extends the scalar function holomorphically across that diameter.
The identity theorem makes it zero in the disk and then throughout the connected domain $\mathcal D_\kappa$.
Since continuous linear functionals separate points, $P_d(z)=0$ on $\mathcal D_\kappa$.
The imaginary-frequency limit \eqref{nf:eq:imaginary-ray} gives
\begin{equation}\label{nf:eq:zero-rays}
 \ray q(x,d)=0,\qquad x\in\partial B_R,\quad d\in\Sph.
\end{equation}
Initially this is an identity in surface $L^2$.
Continuity of the half-rays makes it pointwise.
Chord subtraction \eqref{nf:eq:chord} and the X-ray identity \eqref{nf:eq:Xidentity} then give $q=0$ in $B_R$, without a positive-order Sobolev prior.
Equation~\eqref{nf:eq:riesz-zero} also holds.

Let $h=|D|^{-1}q$.
The Riesz kernel is a nonzero multiple of $|x|^{-2}$, and the Hardy--Littlewood--Sobolev inequality gives $h\in L^2$.
The Riesz transforms of $q\in L^2$ give $\nabla h\in L^2$.
Hence $h\in H^1(\R^3)$ and
\[
 h=0,\qquad |D|h=q=0\quad\text{in }B_R.
\]
Apply \cite[Theorem~1.2]{GhoshSaloUhlmann2020} to $h\in H^1(\R^3)$ with fractional order $\frac{1}{2}$.
Since both $h$ and $(-\Delta)^{\frac{1}{2}}h$ vanish in the nonempty open ball $B_R$, the theorem gives $h=0$ throughout $\R^3$.
Hence $q=|D|h=0$ throughout $\R^3$.
\end{proof}
\subsection{Finite-band continuation and interior stability}\label{nf:sec:stability-proof}

Assume the hypotheses of Theorem~\ref{intro:thm:nearfield} throughout this subsection.
For $z\in\overline{\mathcal D_\kappa}$ define
\[
 D(z)=\sup_{d\in\Sph}|z|
 \norm{a_{V_1,d}(z)-a_{V_2,d}(z)}_{L^2(\partial B_R)}.
\]
For real $k$, this equals the discrepancy at frequency $k$ in \eqref{nf:eq:epsilon}.
Proposition~\ref{nf:prop:ray}, \eqref{nf:eq:chord}, and the hemisphere projection formula $\dd p=|\nu\cdot d|\dd S\le\dd S$, where $\nu(x)=\frac{x}{R}$ is the outward unit normal on $\partial B_R$, yield
\begin{equation}\label{nf:eq:Xestimate}
 \int_{\Sph}\int_{d^\perp}|Xf(p,d)|^2\dd p\dd d
 \le C(D(z)^2+|z|^{-2}).
\end{equation}
There is no division by a grazing angle.

\begin{lemma}\label{nf:lem:extension}
Fix $s>0$ and assume $\norm{q|_{B_R}}_{H^s(B_R)}\le Q_s$.
There is an extension $F$ of $q|_{B_R}$ satisfying
\begin{align}
 \norm F_{H^{-\frac{1}{2}}(\R^3)}^2
 &\le C(D(z)^2+|z|^{-2}),\label{nf:eq:negative}\\
 \norm F_{H^s(\R^3)}&\le CQ_s.\label{nf:eq:positive}
\end{align}
Consequently, for every $\rho\ge1$,
\begin{equation}\label{nf:eq:split}
 \norm q_{L^2(B_R)}^2
 \le C\{\rho(D(z)^2+|z|^{-2})+\rho^{-2s}\}.
\end{equation}
\end{lemma}
\begin{proof}
Let $r=q|_{B_R}$.
By Rychkov's universal extension theorem \cite{Rychkov}, a single operator $E$ extends the restriction spaces and is bounded at both indices $-\frac{1}{2}$ and $s$.
Choose $F=Er$.
The quotient definition gives
\[
 \norm r_{H^{-\frac{1}{2}}(B_R)}
 \le\norm{\mathbf1_{B_R}q}_{H^{-\frac{1}{2}}(\R^3)}.
\]
Thus \eqref{nf:eq:Xestimate} and \eqref{nf:eq:Xidentity} imply
\[
 \norm F_{H^{-\frac{1}{2}}}^2
 \le C\norm{\mathbf1_{B_R}q}_{H^{-\frac{1}{2}}}^2
 \le C\int_{\R^3}\frac{|\widehat{\mathbf1_{B_R}q}(\xi)|^2}{|\xi|}\dd\xi
 \le C(D(z)^2+|z|^{-2}).
\]
The same extension operator gives \eqref{nf:eq:positive} from the local prior.
Since $F=q$ in $B_R$, Plancherel and a split at $|\xi|=\rho$ give \eqref{nf:eq:split}.
This uses the positive regularity of $F$; the zero extension is used only in the negative-order estimate.
\end{proof}

Choose $\rho=|z|^{\frac{2}{1+2s}}$.
On the real axis and in the upper half-plane we obtain
\begin{equation}\label{nf:eq:HF}
 \norm q_{L^2(B_R)}^2
 \le C\{|z|^{a_{\rm N}}D(z)^2+|z|^{-2b_{\rm N}}\},\qquad
 a_{\rm N}=\frac{2}{1+2s},\quad b_{\rm N}=\frac{2s}{1+2s}.
\end{equation}
Here $a_{\rm N}+2b_{\rm N}=2$.
The X-ray estimate controls a full negative Sobolev norm, so $\rho$ has no Fourier-aperture restriction.

By \eqref{nf:weight:physical} and \eqref{nf:eq:epsilon}, the boundary norms of $P_d$ on $[K_0,K]$ are at most $\epsN$.
\begin{lemma}\label{nf:lem:harmonic}
For $K\ge K_0+1$ and $A\ge K$,
\begin{equation}\label{nf:eq:propagation}
 D(iA)\le M\exp\!\left(-c\frac{K}{A}\log\frac{M}{\epsN}\right),
 \qquad 0<\epsN<M,
\end{equation}
where $c>0$ depends only on $K_0,\kappa$.
\end{lemma}
\begin{proof}
The Joukowski map
\[
 J(z)=z+\frac{\kappa^2}{z}
\]
is conformal from $\mathcal D_\kappa$ onto the upper half-plane.
It sends the semicircle to $[-2\kappa,2\kappa]$, the measured interval to $[b_0,b_K]=[J(K_0),J(K)]$, and $iA$ to $iY$, where $Y=A-\frac{\kappa^2}{A}$.
Consequently the harmonic measure $\omega_A$ of the observed interval is
\[
 \omega_A=\frac{1}{\pi}\int_{b_0}^{b_K}\frac{Y}{t^2+Y^2}\dd t.
\]
Since $K_0\ge2\kappa$,
\[
 b_K-b_0=(K-K_0)\left(1-\frac{\kappa^2}{KK_0}\right)
 \ge\frac{3K}{4(K_0+1)},\quad
 b_K\le\frac{5}{4}K,\quad \frac{3}{4}A\le Y\le A .
\]
Thus $\omega_A\ge \frac{cK}{A}$; for instance $c=\frac{9}{41\pi(K_0+1)}$ is admissible.
Apply the two-constants theorem to each norm-one scalarization of $P_d$, with boundary bounds $\epsN$ on the measured interval and $M$ on the rest of the boundary.
This gives
\[
 \norm{P_d(iA)}\le M\left(\frac{\epsN}{M}\right)^{\omega_A}.
\]
The constants are uniform in the scalarization and $d$, so taking suprema proves \eqref{nf:eq:propagation}.
\end{proof}

\begin{proof}[Proof of Theorem~\ref{intro:thm:nearfield}]
Shrink $\varepsilon_*$ so that $L=\log\frac{M}{\epsN}\ge1$.
Put $X=KL$, $H=\log(e+X)$, $a=a_{\rm N}$, and $b=b_{\rm N}$; then $a+2b=2$.
Fix $\gamma>0$ with $\gamma\le1$ and $\frac{2c}{\gamma}\ge a+2b$, and set $A_*=\frac{\gamma X}{H}$.
If $A_*\ge K$, Lemma~\ref{nf:lem:harmonic} and \eqref{nf:eq:HF} imply
\[
 \norm q_{L^2(B_R)}^2
 \le C\{A_*^a M^2e^{-\frac{2cX}{A_*}}+A_*^{-2b}\}
 \le C\left(\frac{H}{X}\right)^{2b}.
\]
Indeed $A_*\le X\le e+X$, so the exponential is absorbed because $\frac{2c}{\gamma}\ge a+2b$.
If $A_*<K$, use the measured real frequency $K$ in \eqref{nf:eq:HF}.
The inequality $A_*<K$ gives $K^{-2b}\le A_*^{-2b}=\gamma^{-2b}(\frac{H}{X})^{2b}$, and hence
\[
 \norm q_{L^2(B_R)}^2
 \le C\left\{K^a\epsN^2+\left(\frac{H}{X}\right)^{2b}\right\}.
\]
Together these cases prove \eqref{intro:eq:nearfield}.
The interval length enters through the factor $\frac{K}{A}$ in Lemma~\ref{nf:lem:harmonic}.
\end{proof}

\section{Far-field stability for exponential tails}\label{ff:sec:exponential}
We first prove Theorem~\ref{ff:thm:exponential} for the class $\mathcal E_{\mu,Q}$ and the discrepancy \eqref{ff:eq:epsilon}.
The proof first obtains a common analytic strip for the individual amplitudes.
It then combines propagation of the measured error with a high-frequency Fourier estimate.
All constants below are uniform over the stated exponential class.

\subsection{Weighted resolvent and amplitude continuation}
\begin{lemma}\label{ff:lem:exp-strip}
There are $b>0$ and $\kappa_{\rm e}>0$ such that $A_V(z,\theta,d)$ is holomorphic and uniformly bounded for
\[
 \operatorname{Re}z>\kappa_{\rm e},\qquad
 |\operatorname{Im}z|<b,
\]
uniformly in $V\in\mathcal E_{\mu,Q}$ and real directions.
For real $k>\kappa_{\rm e}$,
\begin{equation}\label{ff:eq:exp-born}
 \left|A_V(k,\theta,d)
       +\frac{1}{4\pi}\widehat V\bigl(k(\theta-d)\bigr)\right|
 \le Ck^{-1}.
\end{equation}
\end{lemma}

\begin{proof}
Choose $a=\frac{\mu}{4}$, $w(x)=e^{-a\ip{x}}$, and $b=\frac{a}{8}$.
The strong Huygens principle, with each weight split at radius $\frac{t}{4}$, gives
\[
 \norm{w\cos(t|D|)w}_{2\to2}\le2e^{-\frac{at}{4}}.
\]
Consequently the wave formula
\begin{equation}\label{ff:eq:exp-T}
 T(z):=-\frac{1}{iz}\int_0^\infty
 e^{izt}w\cos(t|D|)w\dd t
\end{equation}
is holomorphic for $\operatorname{Im}z>-\frac{a}{4}$, $z\ne0$, and
\[
 \norm{T(z)}_{2\to2}\le C_a|z|^{-1}
 \quad\text{when }|\operatorname{Im}z|\le b.
\]
For $\operatorname{Im}z>0$ it equals $wR_0(z)w$.
Thus \eqref{ff:eq:exp-T} continues the weighted outgoing free resolvent across the real axis in the required strip.

Put $v=\frac{V}{w^2}$ and
\[
 f_d(z,x)=w(x)^{-1}V(x)e^{izx\cdot d}.
\]
Since $2a<\mu$ and $a+b<\mu$, $v$ is uniformly bounded and $f_d(z)$ is uniformly bounded in $L^2$ on the strip.
It is $L^2$ valued holomorphic there, with exponential domination on each slightly larger closed substrip.
For sufficiently large $\operatorname{Re}z$, define
\[
 h_d(z)=-(I+T(z)v)^{-1}T(z)f_d(z).
\]
The inverse is given by a uniformly convergent Neumann series and $\norm{h_d(z)}_2\le C|z|^{-1}$.
At real $k$ it is the weighted scattered field $wu_V^s(k,d)$.

The amplitude is continued by
\begin{align}\label{ff:eq:exp-A}
 A_V(z,\theta,d)=-\frac{1}{4\pi}\biggl\{&
 \int_{\R^3}V(x)e^{izx\cdot(d-\theta)}\dd x
 +\int_{\R^3}\frac{V(x)}{w(x)}e^{-izx\cdot\theta}
                  h_d(z,x)\dd x\biggr\}.
\end{align}
The first integral is holomorphic and bounded since $2b<\mu$.
The second is a continuous bilinear $L^2$ pairing of holomorphic functions and is $O(|z|^{-1})$.
There is no conjugation in this pairing.
These statements prove the asserted bounded holomorphy, and their real boundary values give \eqref{ff:eq:exp-born}.
\end{proof}

\subsection{Continuation and Fourier splitting}
\begin{proof}[Proof of Theorem~\ref{ff:thm:exponential}]
Put $\alpha_{\rm e}=\frac{6}{3+2s}$ and $\beta_{\rm e}=\frac{2s}{3+2s}$.
Write $q=V_1-V_2$ and $D(k)=\sup_{\theta,d}|A_{V_1}(k,\theta,d)-A_{V_2}(k,\theta,d)|$.
From \eqref{ff:eq:exp-born}, for $|\xi|\le2k$,
\[
 |\widehat q(\xi)|\le C D(k)+Ck^{-1}.
\]
Choose $\rho=k^{\frac{2}{3+2s}}\le2k$ for $k\ge1$.
The low Fourier ball contributes $C\rho^3(D(k)^2+k^{-2})$, and the $H^s$ prior contributes $C\rho^{-2s}$ outside that ball.
Consequently
\begin{equation}\label{ff:eq:exp-fourier}
 \norm q_2^2\le C\left(k^{\alpha_{\rm e}}D(k)^2+k^{-2\beta_{\rm e}}\right).
\end{equation}

Let $M_{\rm e}$ bound the difference of the continued amplitudes in a fixed narrower strip.
Slit that strip along $[K_0,K]$.
For real $t\ge K+1$, the harmonic measure of the slit has the lower bound
\begin{equation}\label{ff:eq:strip-measure}
 \omega_{\rm e}(t)\ge c e^{-C_1(t-K)}.
\end{equation}
For completeness, restrict to the translated right half-strip $\operatorname{Re}z>K-1$ with the terminal slit $[K-1,K]$.
At $K+1$ its slit harmonic measure is a fixed positive number.
A Harnack chain along the real axis, with balls of radius fixed by the strip width and $O(t-K+1)$ balls, gives \eqref{ff:eq:strip-measure}.
Domain comparison gives the same lower bound for the original slit.
The two-constants theorem consequently gives
\[
 D(t)\le M_{\rm e}\exp\left(-c L_{\rm e}e^{-C_1(t-K)}\right).
\]

Set $\ell=\log L_{\rm e}$ and fix $0<\nu\le(2C_1)^{-1}$.
Choose the error threshold so that $\ell\ge1$ and $\nu\ell\ge1$.
If $\ell<K$, use \eqref{ff:eq:exp-fourier} at the measured wavenumber $k=K$.
Since $K+\ell<2K$, the remainder is at most $C(K+\ell)^{-2\beta_{\rm e}}$.
If $\ell\ge K$, put $t=K+\nu\ell$.
Then $t\asymp K+\ell$ and
\[
 D(t)\le M_{\rm e}e^{-c\sqrt{L_{\rm e}}}.
\]
Since $K\le\ell$, the exponentially small data term in \eqref{ff:eq:exp-fourier} is bounded uniformly by $C(K+\ell)^{-2\beta_{\rm e}}$ after reducing the error threshold; the second term has the same bound because $t\asymp K+\ell$.
This proves \eqref{ff:eq:exp-main}.
The zero-error conclusion follows by letting an arbitrary positive error upper bound decrease to zero.
\end{proof}

The strip width in Lemma~\ref{ff:lem:exp-strip} depends on the exponential decay rate $\mu$.
The harmonic-measure lower bound in \eqref{ff:eq:strip-measure} decreases exponentially with the distance beyond the observed interval.
Consequently the continuation step accesses frequencies of size $K+c\log\log(\frac{M_{\rm e}}{\epsE})$ at controlled error.
This gives the remainder in \eqref{ff:eq:exp-main}.
No uniform estimate as $\mu\downarrow0$ follows from this argument.

\section{Further discussion}\label{sec:questions}
Finally, we  briefly discuss uniqueness for Schwartz potentials.
In the textbook \emph{The Calder\'on Problem: An Introduction} (\cite[Question~1.4.7]{FeldmanSaloUhlmann2025}), Feldman, Salo, and Uhlmann ask whether equality of the scattering amplitudes on a frequency interval determines a  Schwartz potential.
In the three-dimensional notation of this paper, the question is whether
\[
 A_{V_1}(k,\theta,d)=A_{V_2}(k,\theta,d),\qquad
 k\in[k_-,K],\quad \theta,d\in\Sph,\quad 0<k_-<K,
\]
implies $V_1=V_2$ for arbitrary real $V_1,V_2\in\Sc(\R^3)$.

A condition on the difference gives a partial answer by combining two classical results.
Novikov--Henkin's Fourier-ball recovery \cite[Corollary~2.5a]{NovikovHenkin1987} gives $\widehat q(\xi)=0$ for $|\xi|<2K$, where $q=V_1-V_2$.
If, in addition,
\[
 \sum_{j=1}^{\infty}\norm{|x|^{2j}q}_{L^2(\R^3)}^{-\frac{1}{2j}}=\infty,
\]
then Plancherel identifies this condition with Chernoff's divergence condition for the $L^2$ norms of the Laplacian iterates of $\widehat q$.
All derivatives of $\widehat q$ vanish at the origin, so Chernoff's theorem \cite[Theorem~6.1]{Chernoff1975} gives $\widehat q=0$ everywhere and hence $V_1=V_2$; a zero moment norm already implies $q=0$.
See also \cite[Theorem~1.1]{ChernoffEuclidean}.

In this paper, our main focus is on stability for inverse potential scattering from multi-wavenumber data. The uniqueness problem for arbitrary real Schwartz potentials posed in Question~1.4.7  remains open.

\appendix
\section{Weighted resolvent and local field estimates}\label{app:weighted}

\begin{proof}[Proof of Lemma~\ref{nf:weight:kernel}]
The Hardy--Littlewood--Sobolev inequality gives the integrable majorant
\[
 \iint\frac{w(x)^2w(y)^2}{|x-y|^2}\dd x\dd y
 \le C\norm{w^2}_{L^{\frac{3}{2}}}^2<\infty.
\]
The nonnegative delay makes the exponential no larger than one in modulus.
This proves the uniform Hilbert--Schmidt bound and boundary continuity.
On compact subsets of the open upper half-plane, the factors $\tau_d^j e^{-\operatorname{Im}z\,\tau_d}$ are bounded; the same majorant proves holomorphy.

For $\operatorname{Im}\zeta>0$, integration by parts in the free wave representation gives
\[
 wR_0(\zeta)w
 =-\frac{1}{i\zeta}\int_0^\infty
 e^{i\zeta t}w\cos(t|D|)w\dd t.
\]
The boundary term at infinity vanishes because $\norm{w\sin(t|D|)|D|^{-1}w}\le t$ and the exponential is damped.
By \eqref{nf:weight:wave}, dominated convergence takes this identity to every real $k\ne0$ and yields $\norm{wR_0(k)w}\le \frac{J_w}{|k|}$.
The Hilbert--Schmidt majorant identifies the limit with the free Green kernel; negative $k$ gives its incoming counterpart.
Real modulation is unitary, and hence
\[
 \mathcal B_d(k)=e^{-ikx\cdot d}wR_0(k)w e^{ikx\cdot d},
 \qquad \norm{k\mathcal B_d(k)}\le J_w.
\]
For $a>0$, the operator function $\frac{z\mathcal B_d(z)}{1-iaz}$ is bounded and holomorphic in the upper half-plane.
For $\operatorname{Im}z\ge0$, $|1-iaz|^2=(1+a\operatorname{Im}z)^2+a^2(\operatorname{Re}z)^2 \ge a^2|z|^2$, so the regularized function is bounded by the Hilbert--Schmidt estimate.
Its denominator vanishes only at $-\frac{i}{a}$.
Its real boundary norm is at most $J_w$.
Apply the bounded analytic maximum principle to scalar pairings with unit vectors and then let $a\downarrow0$.
This proves \eqref{nf:weight:highenergy}.
\end{proof}

\begin{proof}[Proof of Lemma~\ref{nf:weight:trace}]
Fix a coordinate vector $e_j$, and put $V_\eta(x)=V(x+\eta e_j)$, $v_\eta=\frac{V_\eta}{w^2}$, and $f_\eta=\frac{V_\eta}{w}$.
The fundamental theorem along almost every coordinate line, the derived weighted bound \eqref{nf:eq:weightedclass}, and \eqref{nf:weight:translation} give
\begin{equation}\label{nf:weight:diffsources}
 \norm{v_\eta-v}_\infty+\norm{f_\eta-f_V}_2
 \le C Q|\eta|,\qquad |\eta|\le1.
\end{equation}
Only the first derivatives are used.
The translated equations have a common Neumann bound.
With $B=\mathcal B_d(z)$ and $g_\eta=g_{V_\eta,d}(z)$, subtraction gives
\[
 (I+Bv_\eta)(g_\eta-g_0)
 =-B\bigl(f_\eta-f_V+(v_\eta-v)g_0\bigr),
\]
and consequently
\begin{equation}\label{nf:weight:diffg}
 \norm{g_\eta-g_0}_2\le C|\eta||z|^{-1}.
\end{equation}
The unweighted modulated kernel is translation invariant.
Condition \eqref{nf:weight:translation} ensures that translation preserves the weighted $L^2$ class.
Uniqueness therefore identifies $g_\eta(x)=w(x)a_{V,d}(z,x+\eta e_j)$.
Difference quotients in \eqref{nf:weight:diffg}, followed by the local lower bound on $w$, prove \eqref{nf:weight:Hone}.

The holomorphic $L^2$ field is holomorphic as a local distribution.
Spatial differentiation preserves this property.
Together with the locally uniform $H^1$ bound, weak holomorphy and local boundedness give holomorphy into $H^1(U)$.
Indeed, pairings against smooth test functions for the field and its first derivatives are holomorphic; density and the local $H^1$ bound extend this to continuous $H^1$ functionals.
At a real boundary point the fields converge in local $L^2$, while their $H^1$ norms stay bounded.
Interpolation gives convergence in $H^t(U)$ for every $0<t<1$.
Choosing $t>\frac{1}{2}$ and taking the trace proves the claimed boundary continuity on the sphere.

For $r_{V,d}=a_{V,d}-b_{V,d}$, the integral equation gives
\[
 wr_{V,d}=-Bv g_{V,d},\qquad
 \norm{wr_{V,d}}_2\le C|z|^{-2}.
\]
Subtracting this identity for $V_\eta$ and $V$ yields
\[
 \norm{w(r_{V_\eta,d}-r_{V,d})}_2
 \le\norm B\bigl(\norm{v_\eta-v}_\infty\norm{g_\eta}_2
             +\norm v_\infty\norm{g_\eta-g_0}_2\bigr)
 \le C|\eta||z|^{-2}.
\]
Translation invariance identifies the left side with a spatial difference of $r_{V,d}$.
The same difference-quotient argument proves \eqref{nf:weight:nonlinear}, and the trace theorem gives the corresponding $L^2(\partial B_R)$ bound.
\end{proof}

\section{Half-ray profiles}\label{app:rays}
The estimates below supply both the quantitative half-ray remainder and the imaginary-frequency limit for exact uniqueness.
\begin{lemma}\label{nf:reg:proposition}
Let $m>2$ and $R>0$.
If $V\in C^{2,\beta}_m$, $0<\beta\le1$, then
\begin{equation}\label{nf:reg:secondorder}
 \sup_{\substack{x\in\overline{B_R}\\d\in\Sph}}
 |2iz\,b_{V,d}(z,x)-\ray V(x,d)|
 \le C_{m,\beta,R}|z|^{-1}\norm V_{C^{2,\beta}_m},
 \quad \operatorname{Im}z\ge0,\quad |z|\ge1.
\end{equation}
If $V\in C^1_m$, then
\begin{equation}\label{nf:reg:C1imaginary}
 \sup_{\substack{x\in\overline{B_R}\\d\in\Sph}}
 |-2A\,b_{V,d}(iA,x)-\ray V(x,d)|
 \le C_{m,R}\norm V_{C^1_m}A^{-\frac{1}{2}},\qquad A\ge1.
\end{equation}
\end{lemma}
\begin{proof}
We first record a bound that uses only one derivative.
Rotate the coordinates so that $d=e_3$ and set $r=x-y$.
The parabolic coordinates
\[
 u=|r|-r\cdot d,\qquad v=|r|+r\cdot d,\qquad
 r=\sqrt{uv}\,\omega_\phi+\frac{(v-u)d}{2}
\]
have $\frac{\dd r}{|r|}=\frac{1}{2}\dd u\dd v\dd\phi$, where $\omega_\phi$ is a unit vector in $d^\perp$.
Consequently
\begin{equation}\label{nf:reg:bornrepresentation}
 \begin{aligned}
 b_{V,d}(z,x)&=-\frac{1}{8\pi}\int_0^\infty e^{izu}H_{x,d}(u)\dd u,\\
 H_{x,d}(u)&=\int_0^\infty\int_0^{2\pi}
 V\bigl(x-\sqrt{uv}\,\omega_\phi-\frac{(v-u)d}{2}\bigr)
 \dd\phi\dd v.
 \end{aligned}
\end{equation}
In particular, $H_{x,d}(0)=4\pi \ray V(x,d)$.
For the trajectory $y=x-\sqrt{uv}\,\omega_\phi-\frac{(v-u)d}{2}$,
\[
 \ip y\ge c_R\ip{u+v},\qquad
 |y_u|\le C(1+\sqrt{\frac{v}{u}}).
\]
It follows that
\begin{equation}\label{nf:reg:C1profile}
 \norm{H'}_{L^1(0,\infty)}
 \le C_R\norm V_{C^1_m}
 \int_0^\infty\int_0^\infty
 (1+\sqrt{\frac{v}{u}})\ip{u+v}^{-m}\dd v\dd u
 \le C_{m,R}\norm V_{C^1_m}.
\end{equation}
For the last inequality, use $r=u+v$ and $a=\frac{u}{u+v}$.
The angular factor $1+\sqrt{\frac{1-a}{a}}$ is integrable on $(0,1)$, and the remaining integral is $\int_0^\infty r\ip r^{-m}\dd r$.
The same estimates show that $H\in W^{1,1}(0,\infty)$, $H(u)\to0$, and its trace at zero equals the value above.
Thus one integration by parts gives
\begin{equation}\label{nf:reg:C1remainder}
 2iz\,b_{V,d}(z,x)-\ray V(x,d)
 =\frac{1}{4\pi}\int_0^\infty e^{izu}H'(u)\dd u,
 \qquad
 |2iz\,b_{V,d}-\ray V|\le C_{m,R}\norm V_{C^1_m}.
\end{equation}
These assertions hold throughout the closed upper half-plane for $z\ne0$; all boundary terms follow from the stated integrability.

For the qualitative limit on the imaginary axis we need a uniform majorant, beyond the $L^1$ bound in \eqref{nf:reg:C1profile}.
The trajectory estimates above give
\begin{equation}\label{nf:reg:C1majorant}
 \sup_{x\in\overline{B_R},d\in\Sph}|H'_{x,d}(u)|
 \le C_{m,R}\norm V_{C^1_m}
 \begin{cases}
  1+u^{-\frac{1}{2}},&0<u\le1,\\
  u^{1-m},&u\ge1.
 \end{cases}
\end{equation}
For $u\le1$, integrate $(1+\sqrt{\frac{v}{u}})\ip v^{-m}$ in $v$.
For $u\ge1$, the substitution $v=ut$ gives the second bound; both integrals converge for $m>2$.
Applying \eqref{nf:reg:C1majorant} to \eqref{nf:reg:C1remainder} gives \eqref{nf:reg:C1imaginary}.

We next prove \eqref{nf:reg:secondorder}.
Define the normalized circular means
\[
 \mathcal N(t,r)=\frac{1}{2\pi}\int_0^{2\pi}
 V(x-td-r\omega_\phi)\dd\phi,
 \qquad \mathcal M(t,\rho)=\mathcal N(t,\sqrt\rho).
\]
For $0<u\le1$, put $t=\frac{v-u}{2}$ and $r=\sqrt{uv}$.
At $0<r\le1$, the cancellation of the angular average of $\nabla V(x-td)\cdot\omega_\phi$ yields
\begin{align*}
 \mathcal N_{rr}(t,r)-\frac{\mathcal N_r(t,r)}{r}
 &=\frac{1}{2\pi}\int_0^{2\pi}\int_0^1
 \bigl(D^2V(x-td-r\omega_\phi)\\
 &\qquad{}-D^2V(x-td-\theta r\omega_\phi)\bigr)
       [\omega_\phi,\omega_\phi]\dd\theta\dd\phi.
\end{align*}
The local weighted H\"older condition therefore gives
\[
 |\mathcal M_{\rho\rho}(t,uv)|
 =\frac{|\mathcal N_{rr}(t,r)-\frac{\mathcal N_r(t,r)}{r}|}{4r^2}
 \le C\norm V_{C^{2,\beta}_m}\ip v^{-m}r^{\beta-2}.
\]
The points in this disk have weight comparable to $\ip v$ when $v$ is large; bounded $v$ contributes only to the constant.
For $r\ge1$, the direct identity for $\mathcal M_{\rho\rho}$ instead gives the bound $C\norm V_{C^2_m}\ip v^{-m}r^{-2}$.
Also
\[
 |\mathcal M_{t\rho}(t,uv)|\le C\norm V_{C^2_m}\ip v^{-m}r^{-1},
 \qquad
 |\mathcal M_{tt}(t,uv)|\le C\norm V_{C^2_m}\ip v^{-m}.
\]
The disk formula for $\mathcal M_\rho$, or the same angular cancellation, gives $|\mathcal M_\rho|\le C\norm V_{C^2_m}\ip v^{-m}$ also at $r=0$.
Hence $H'(0)$ exists and is bounded when $m>2$.

The second derivative along $u\mapsto(\frac{v-u}{2},uv)$ is
\[
 \frac{1}{4}\mathcal M_{tt}-v\mathcal M_{t\rho}
       +v^2\mathcal M_{\rho\rho}.
\]
Its mixed term is integrable since $\int_0^1v(uv)^{-\frac{1}{2}}\dd u=2\sqrt v$.
For $v\ge1$, the small-radius term satisfies
\[
 \int_0^{\frac{1}{v}}v^2(uv)^{-1+\frac{\beta}{2}}\dd u=\frac{2}{\beta} v,
\]
and the remaining radius range contributes $\int_{\frac{1}{v}}^1\frac{v}{u}\dd u=v\log v$.
The contribution of $0\le v\le1$ is bounded.
It follows that
\[
 \int_0^1|H''(u)|\dd u
 \le C_{\beta,R}\norm V_{C^{2,\beta}_m}
 \int_0^\infty(1+\sqrt v+v+v\log(1+v))\ip v^{-m}\dd v<\infty.
\]
For $u\ge1$, direct differentiation of the trajectory gives $|y_{uu}|\le C\sqrt v\,u^{-\frac{3}{2}}$, and hence
\[
 |H''(u)|\le C_R\norm V_{C^2_m}
 \int_0^\infty\left(1+\frac{v}{u}+\frac{\sqrt v}{u^{\frac{3}{2}}}\right)
 \ip{u+v}^{-m}\dd v
 \le C_{m,R}\norm V_{C^2_m}u^{1-m}.
\]
Thus $H''\in L^1(0,\infty)$ with the claimed uniform bound.
The corresponding first-derivative estimate shows $H'(u)\to0$.
One further integration by parts in \eqref{nf:reg:C1remainder} proves \eqref{nf:reg:secondorder}.

\end{proof}


\begin{thebibliography}{99}
\small\raggedright

\bibitem{BaoLiLinTriki2015}
G. Bao, P. Li, J. Lin, and F. Triki,
\emph{Inverse scattering problems with multi-frequencies},
Inverse Problems \textbf{31} (2015), no.~9, 093001.
\href{https://doi.org/10.1088/0266-5611/31/9/093001}{doi:10.1088/0266-5611/31/9/093001}.

\bibitem{BarkerEtAl2022}
D. S. Barker, B. P. Acharya, J. A. Fedchak, N. N. Klimov,
E. B. Norrgard, J. Scherschligt, E. Tiesinga, and S. P. Eckel,
\emph{Precise quantum measurement of vacuum with cold atoms},
Rev. Sci. Instrum. \textbf{93} (2022), no.~12, 121101.
\href{https://doi.org/10.1063/5.0120500}{doi:10.1063/5.0120500}.

\bibitem{BZZ}
G. Bao, J. Zhai, and Y. Zhao,
\emph{Stability for the inverse potential scattering problems with multi-wavenumbers},
SIAM J. Math. Anal. \textbf{58} (2026), no.~3, 2111--2137.
\href{https://doi.org/10.1137/25M1767108}{doi:10.1137/25M1767108}.

\bibitem{ChernoffEuclidean}
M.~Bhowmik, S.~Pusti, and S.~K.~Ray,
\emph{A theorem of Chernoff on quasi-analytic functions for Riemannian
symmetric spaces},
Int. Math. Res. Not. IMRN (2023), no.~13, 10942--10963.
\href{https://doi.org/10.1093/imrn/rnab143}{doi:10.1093/imrn/rnab143}.

\bibitem{Chernoff1975}
P. R. Chernoff,
\emph{Quasi-analytic vectors and quasi-analytic functions},
Bull. Amer. Math. Soc. \textbf{81} (1975), no.~4, 637--646.
\href{https://doi.org/10.1090/S0002-9904-1975-13806-7}{doi:10.1090/S0002-9904-1975-13806-7}.

\bibitem{EdwardsEtAl2017}
J. P. Edwards, U. Gerber, C. Schubert, M. A. Trejo, and A. Weber,
\emph{The Yukawa potential: ground state energy and critical screening},
Prog. Theor. Exp. Phys. (2017), no.~8, 083A01.
\href{https://doi.org/10.1093/ptep/ptx107}{doi:10.1093/ptep/ptx107}.

\bibitem{EskinRalston1995}
G. Eskin and J. Ralston,
\emph{Inverse scattering problem for the Schr\"odinger equation with
magnetic potential at a fixed energy},
Comm. Math. Phys. \textbf{173} (1995), no.~1, 199--224.
\href{https://doi.org/10.1007/BF02100187}{doi:10.1007/BF02100187}.

\bibitem{FeldmanSaloUhlmann2025}
J. Feldman, M. Salo, and G. Uhlmann,
\emph{The Calder\'on Problem: An Introduction},
Graduate Studies in Mathematics, vol.~253,
Amer. Math. Soc., Providence, RI, 2025.
\href{https://doi.org/10.1090/gsm/253}{doi:10.1090/gsm/253}.

\bibitem{GhoshSaloUhlmann2020}
T. Ghosh, M. Salo, and G. Uhlmann,
\emph{The Calder\'on problem for the fractional Schr\"odinger equation},
Anal. PDE \textbf{13} (2020), no.~2, 455--475.
\href{https://doi.org/10.2140/apde.2020.13.455}{doi:10.2140/apde.2020.13.455}.

\bibitem{GribakinFlambaum1993}
G. F. Gribakin and V. V. Flambaum,
\emph{Calculation of the scattering length in atomic collisions using
the semiclassical approximation},
Phys. Rev. A \textbf{48} (1993), no.~1, 546--553.
\href{https://doi.org/10.1103/PhysRevA.48.546}{doi:10.1103/PhysRevA.48.546}.

\bibitem{Isaev2013}
M. I. Isaev,
\emph{Exponential instability in the inverse scattering problem on the energy interval},
Funct. Anal. Appl. \textbf{47} (2013), no.~3, 187--194.
\href{https://doi.org/10.1007/s10688-013-0025-9}{doi:10.1007/s10688-013-0025-9}.

\bibitem{Isakov2011}
V. Isakov,
\emph{Increasing stability for the Schr\"odinger potential from the
Dirichlet-to-Neumann map},
Discrete Contin. Dyn. Syst. Ser. S \textbf{4} (2011), no.~3, 631--640.
\href{https://doi.org/10.3934/dcdss.2011.4.631}{doi:10.3934/dcdss.2011.4.631}.

\bibitem{IsakovLuXu2020}
V. Isakov, S. Lu, and B. Xu,
\emph{Linearized inverse Schr\"odinger potential problem at a large wavenumber},
SIAM J. Appl. Math. \textbf{80} (2020), no.~1, 338--358.
\href{https://doi.org/10.1137/18M1226932}{doi:10.1137/18M1226932}.

\bibitem{IsakovNagayasuUhlmannWang2014}
V. Isakov, S. Nagayasu, G. Uhlmann, and J.-N. Wang,
\emph{Increasing stability of the inverse boundary value problem
for the Schr\"odinger equation},
in \emph{Inverse Problems and Applications},
Contemp. Math. \textbf{615}, Amer. Math. Soc., Providence, RI,
2014, 131--141.
\href{https://doi.org/10.1090/conm/615/12268}{doi:10.1090/conm/615/12268}.

\bibitem{LiZhaiZhao2020}
P. Li, J. Zhai, and Y. Zhao,
\emph{Stability for the acoustic inverse source problem in inhomogeneous media},
SIAM J. Appl. Math. \textbf{80} (2020), no.~6, 2547--2559.
\href{https://doi.org/10.1137/20M1334267}{doi:10.1137/20M1334267}.

\bibitem{NovikovHenkin1987}
R. G. Novikov and G. M. Henkin,
\emph{The $\bar\partial$-equation in the multidimensional inverse scattering problem},
Russian Math. Surveys \textbf{42} (1987), no.~3, 109--180.
\href{https://doi.org/10.1070/RM1987v042n03ABEH001419}{doi:10.1070/RM1987v042n03ABEH001419}.

\bibitem{Novikov1994}
R. G. Novikov,
\emph{The inverse scattering problem at fixed energy for the three-dimensional
Schr\"odinger equation with an exponentially decreasing potential},
Comm. Math. Phys. \textbf{161} (1994), no.~3, 569--595.
\href{https://doi.org/10.1007/BF02101933}{doi:10.1007/BF02101933}.

\bibitem{Rychkov}
V. S. Rychkov,
\emph{On restrictions and extensions of the Besov and Triebel--Lizorkin spaces with respect to Lipschitz domains},
J. London Math. Soc. (2) \textbf{60} (1999), no.~1, 237--257.
\href{https://doi.org/10.1112/S0024610799007723}{doi:10.1112/S0024610799007723}.

\bibitem{Stefanov1990}
P. Stefanov,
\emph{Stability of the inverse problem in potential scattering at fixed energy},
Ann. Inst. Fourier \textbf{40} (1990), no.~4, 867--884.
\href{https://doi.org/10.5802/aif.1239}{doi:10.5802/aif.1239}.

\bibitem{UhlmannVasyExp}
G. Uhlmann and A. Vasy,
\emph{Fixed energy inverse problem for exponentially decreasing potentials},
Methods Appl. Anal. \textbf{9} (2002), no.~2, 239--248.
\href{https://doi.org/10.4310/MAA.2002.v9.n2.a2}{doi:10.4310/MAA.2002.v9.n2.a2}.

\end{thebibliography}
\end{document}